\documentclass[12pt]{article}
\usepackage{mathrsfs}
\usepackage{amsfonts}
\usepackage{amsmath,amsthm,amssymb}

\newcommand{\vs}{\vspace}

\numberwithin{equation}{section} \theoremstyle{plain}
\newtheorem{theorem}{Theorem}[section]
\newtheorem{lemma}{Lemma}[section]

\newtheorem{proposition}{Proposition}[section]

\title{Ground state solutions to a coupled nonlinear logarithmic Hartree system}
\author{ Qihan  He \\
{\footnotesize  School of Mathematics and Information Science, Guangxi University, China}\\
 {\footnotesize  E-mail\,$:$ heqihan277@gxu.edu.cn } \vs{.3cm}\\
 Yafei Li \thanks {Corresponding author} \\
{\footnotesize  School of Mathematics and Information Science, Guangxi University, China}\\
 {\footnotesize E-mail\,$:$ yafeili0906@163.com} \vs{.3cm}\\
Yanfang Peng\\
{\footnotesize  School of Mathematical Sciences, Guizhou Normal University, China}\\  
{\footnotesize  E-mail\,$:$ pyfang2005@sina.com}\\}

\date{}

\begin{document}
 \maketitle 
\begin{abstract}
In this paper, we study the following coupled nonlinear logarithmic Hartree system
\begin{align*}
\left\{ \displaystyle \begin{array}{ll} \displaystyle
-\Delta u+ \lambda_1 u
=\mu_1\left( -\frac{1}{2\pi}\ln(|x|) \ast u^{2} \right)u+\beta \left( -\frac{1}{2\pi}\ln(|x|) \ast v^{2} \right)u,
& x \in \ \mathbb R^2,   \vspace{.4cm}\\ \displaystyle
-\Delta v+ \lambda_2 v
=\mu_2\left( -\frac{1}{2\pi}\ln(|x|) \ast v^{2} \right)v +\beta\left( -\frac{1}{2\pi}\ln(|x|) \ast u^{2} \right)v,
& x \in \ \mathbb R^2,
\end{array}
\right.\hspace{1cm}
\end{align*}
where $\beta, \mu_{i}, \lambda_{i} \ (i=1,2)$ are   positive constants, $\ast$ denotes  the convolution in $\mathbb R^{2}$. By considering the constraint  minimum  problem on the  Nehari manifold, we prove  the existence of  ground state solutions for   $\beta>0$ large enough. Moreover, we also show  that every positive solution  is radially symmetric and decays exponentially.
\end{abstract}

{\footnotesize {\bf   Keywords:} Hartree system, Logarithmic convolution potential, Ground state solution, Radial symmetry.

 }

\section{Introduction}

The time-dependent system of coupled nonlinear Hartree  system   can  be written  as follows:
\begin{align}\label{E:1.1}
\left\{ \displaystyle \begin{array}{ll} \displaystyle
-i\partial_{t} \mathbf{\Psi}_1
=\Delta \mathbf{\Psi}_1
+\mu_1\big(K(x)\ast|\mathbf{\Psi}_1|^{2}\big)\mathbf{\Psi}_1
+\beta\big(K(x)\ast|\mathbf{\Psi}_2|^{2}\big)\mathbf{\Psi}_1,
&(t,x)\in \mathbb R^+\times \mathbb R^N, \vspace{.2cm}\\
-i\partial_{t} \mathbf{\Psi}_2
=\Delta \mathbf{\Psi}_2
+\mu_2\big(K(x)\ast|\mathbf{\Psi}_2|^{2}\big)\mathbf{\Psi}_2
+\beta\big(K(x)\ast|\mathbf{\Psi}_1|^{2}\big)\mathbf{\Psi}_2,
&(t,x)\in \mathbb R^+\times \mathbb R^N,
\end{array}\right.
\end{align}
where $\mathbf{\Psi}_{j} : \mathbb R^+\times \mathbb R^N \to \mathbb C$, $i$ is the imaginary unit, $\mu_1,\mu_2\neq 0$, and $\beta\neq 0$ is a coupling constant  which describes the scattering length of the attractive or repulsive interaction, $K(x)$ is a response function which possesses information on the mutual interaction between the particles.
This system \eqref{E:1.1} appears in several physical models, for instance binary mixtures of Bose--Einstein condensates, or the propagation of mutually incoherent wave packets in nonlinear optics(see \cite{FS, EB, MM, VLK}).
And if ones want to know more about the physical background and mathematical derivation of Hartree's theory in the case of a single equation, we refer readers to \cite{ MPN,ML} and the references therein.

It is well-known that $(\mathbf{\Psi}_1(t,x),\mathbf{\Psi}_2(t,x)):=(e^{i\lambda_{1}t}u(x),e^{i\lambda_{2}t}v(x))$ is
a solitary wave solution of system \eqref{E:1.1} if and only if $(u,v)$ solve the following elliptic system
\begin{align}\label{E:1.2}
\left\{ \displaystyle\begin{array}{ll} \displaystyle
-\Delta u+ \lambda_1 u
=\mu_1\big(K(x)\ast u^{2}\big)u+\beta \big(K(x)\ast v^{2}\big)u ,
&\mbox{in}\ \mathbb R^N, \vspace{.2cm}\\ \displaystyle
-\Delta v+ \lambda_2 v
=\mu_2\big(K(x)\ast v^{2}\big)v+\beta \big(K(x)\ast u^{2}\big)v ,
&\mbox{in}\ \mathbb R^N.
\end{array}\right.
\end{align}
If the response function is the delta function, i.e. $K(x) =\delta(x)$,
then $\eqref{E:1.2}$ turns to the following coupled nonlinear Schr\"{o}dinger system
\begin{align}\label{E:1.3}
\left\{ \displaystyle\begin{array}{ll} \displaystyle
-\Delta u+ \lambda_1 u
=\mu_1 u^{3}+\beta  v^{2} u ,
&\mbox{in}\ \mathbb R^N,   \vspace{.2cm}\\
-\Delta v+ \lambda_2 v
=\mu_2  v^{3} +\beta u^{2} v ,
&\mbox{in}\ \mathbb R^N.
\end{array}
\right.\hspace{1cm}
\end{align}
For the system \eqref{E:1.3}, there are some significant progress on the multiplicity and properties of solutions,
see \cite{AC,BW,CLZ,CZ3,CZ2,LL,LWJ,LW,MZ,SB} and the references therein.

One can see that the fundamental solution to the laplace operator can be denoted as follows:
$$ \Gamma_N (x)=\left\{
\begin{aligned}
&- \frac{1}{2\pi}\ln(|x|), ~~~~~~~~~~~~  N=2;\\
&\frac{1}{N(N-2)w_{N}}|x|^{2-N}, ~~      N\geq 3,
\end{aligned}
\right.
$$
where $w_{N}$ is  the volume of the unit ball in $\mathbb R^N$.
If $K(x)=\Gamma_N (x)$ and  $N\geq3$, then  system \eqref{E:1.2} can be written as
\begin{align}\label{E:1.4}
\left\{ \displaystyle\begin{array}{ll} \displaystyle
-\Delta u+ \lambda_1 u
=\mu_1\big(\int_{\mathbb R^N} \frac{u^2(y)}{|x-y|^{N-2}} \mathrm{d}y \big)u+\beta \big(\int_{\mathbb R^N} \frac{v^2(y)}{|x-y|^{N-2}} \mathrm{d}y \big)u ,
&\mbox{in}\ \mathbb R^N, \vspace{.2cm}\\ \displaystyle
-\Delta v+ \lambda_2 v
=\mu_2\big(\int_{\mathbb R^N} \frac{v^2(y)}{|x-y|^{N-2}} \mathrm{d}y \big)v+\beta \big(\int_{\mathbb R^N} \frac{u^2(y)}{|x-y|^{N-2}} \mathrm{d}y \big)v ,
&\mbox{in}\ \mathbb R^N,
\end{array}
\right.\hspace{1cm}
\end{align}
which is a nonlocal problem and  has been  studied  extensively (See \cite{WS2,WS,YWD,YZW}).

If $K(x)=\Gamma_N (x)$ and  $N=2$, then  system \eqref{E:1.2} becomes the following problem
\begin{align}\label{E:1.5}
\left\{ \displaystyle\begin{array}{ll} \displaystyle
-\Delta u+ \lambda_1 u
=\mu_1\big(-\frac{1}{2\pi}\ln(|x|)\ast u^{2}\big)u+\beta \big(-\frac{1}{2\pi}\ln(|x|)\ast v^{2}\big)u ,
&\mbox{in}\ \mathbb R^2, \vspace{.2cm}\\ \displaystyle
-\Delta v+ \lambda_2 v
=\mu_2\big(-\frac{1}{2\pi}\ln(|x|)\ast v^{2}\big)v+\beta \big(-\frac{1}{2\pi}\ln(|x|)\ast u^{2}\big)v ,
&\mbox{in}\ \mathbb R^2.
\end{array}\right.
\end{align}
When $\beta=0$, studying \eqref{E:1.5} is equivalent to  studying  the following Schr\"{o}dinger--Poisson system
\begin{align}\label{E:1.6}
-\Delta u+ \lambda u
+\mu\Big(\int_{\mathbb R^{2}}\frac{1}{2\pi}\ln(|x-y|)u^{2}(y)\mathrm{d}y \Big)u=0,~~~~ \mbox{in}\ \mathbb R^2.
\end{align}
Since  the integral kernel $\ln(|x|)$ is sign--changing in $\mathbb R^2$,
system \eqref{E:1.6} attracts many researchers'  attention \cite{BM,BCV,CDZ,CT2,CT,CW,DW,LRTZ,SJ}.
To study system \eqref{E:1.6},
Stubble \cite{SJ}  first set up a variational framework  and  proved that if $\lambda\geq0$ and $\mu >0$,
then  the system \eqref{E:1.6} has a unique ground state solution, which is a positive spherically symmetric decreasing function.
Later, Bonheure, Cingolani and Van Schaftingen \cite{BCV} proved the nondegeneracy and  the exponential
decay property of the unique ground state solution to \eqref{E:1.6} with $\lambda>0,\mu=1$.
Cingolani and Weth \cite{CW} considered system \eqref{E:1.6} with a local nonlinear term, i.e.,
\begin{align}\label{E:1.7}
-\Delta u+ \lambda u
+\Big(\int_{\mathbb R^{2}}\frac{1}{2\pi}\ln(|x-y|)u^{2}(y)\mathrm{d}y \Big)u=b|u|^{p-2},~~~~ \mbox{in}\ \mathbb R^2,
\end{align}
where $b\geq0$, $p>2$ and $\lambda\in L^{\infty}(\mathbb R^2)$, and proved that  if $p\geq4$,
then  the problem \eqref{E:1.7} has a sequence of solution pairs $\pm u$ and a ground sate solution.
In addition, the authors also  showed that every positive solution is radially symmetric and monotone decreasing for $p>2$ and  $\lambda >0$ by moving plane method.
Later on, Du and Weth \cite{DW} studied the case of $2<p<4$ and $\lambda=1$,
and proved the existence of ground state solutions and infinitely many nontrivial sign--changing solutions.
In \cite{CT2,CT}, Chen and Tang considered a more general  case  related to \eqref{E:1.7} with axially symmetric potential function
and  general local nonlinearities,   and found a ground state solution in the axially symmetric functions space.
Recently,  Bernini and Mugnai \cite{BM} studied the existence of radially symmetric solutions for \eqref{E:1.6} with a local nonlinear  term,
which does not satisfy the Ambrosetti-Rabinowitz condition.

Motivated by the above mentioned  papers, here  we want to discuss the existence of positive ground stated solutions
and the properties of positive solutions to \eqref{E:1.5} with $ \beta, \lambda_{i}, \mu_{i}>0 (i=1,2)$.
The energy functional corresponding to \eqref{E:1.5} is defined by
\begin{align}\label{E:1.8}
\mathcal J(u,v)=\frac{1}{2}\int_{\mathbb R^2}\big(|\nabla u|^{2}+|\nabla v|^{2}+\lambda_1 u^{2}+ \lambda_2 v^{2}\big) \mathrm{d}x +\frac{1}{4}A_0(u,v),
\end{align}
where
$$ A_{0}(u,v):=\int_{\mathbb R^2}\int_{\mathbb R^2}\frac{1}{2\pi}\ln(|x-y|)\big(\mu_1 u^{2}(x)u^{2}(y) +\mu_2 v^{2}(x)v^{2}(y) + 2\beta u^{2}(x) v^{2}(y)\big)\mathrm{d}x\mathrm{d}y. $$
Note that $\mathcal J$ is not well-defined on $H^1 (\mathbb R^2)\times H^1 (\mathbb R^2)$ even if $ \beta, \lambda_{i},  \mu_{i}>0 (i=1,2)$.
Inspired by \cite{CW} and \cite{SJ}, we define a smaller Hilbert space
\begin{align}\label{E:1.9}
X:=\bigg\{(u,v) \in H:\int_{\mathbb R^{2}} \big( \ln(1+|x|)u^2+\ln(1+|x|)v^2 \big) \mathrm{d}x < \infty \bigg\},
\end{align}
equipped with the norm
$$\displaystyle \|(u,v)\|_{X}^{2}:=\int_{\mathbb R^{2}} \big(|\nabla u|^2+|\nabla v|^2+\lambda_1 u^2+ \lambda_2 v^2 \big)\mathrm{d}x + \int_{\mathbb R^2}\big( \ln(1+|x|)u^2 + \ln(1+|x|)v^2 \big) \mathrm{d}x,$$
where $H:=H^{1}(\mathbb R^{2})\times H^{1}(\mathbb R^{2})$, endowed with the norm
$$\|(u,v)\|_{H}^{2}:=\int_{\mathbb R^2} \big(|\nabla u|^2+|\nabla v|^2+\lambda_1 u^2+ \lambda_2 v^2 \big)\mathrm{d}x.$$
Due to the  Hardy-Littlewood-Sobolev inequality and the following decomposition
$$\ln(|x-y|)=\ln(1+|x-y|)-\ln(1+\frac{1}{|x-y|}),$$
we have that $\mathcal J$ is well-defined and of class $C^{1}$ on $X$.
Moreover, any critical point of   $\mathcal J $ in $X$ corresponds to a solution of \eqref{E:1.5}.

Before stating our result, we give some definitions.
A solution $(u,v)$ of \eqref{E:1.5} is called a nontrivial solution if $u\neq 0$ and $ v\neq 0$,
and a nontrivial solution $(u,v)$ is positive if $u>0,v>0$.
Moreover, we say a solution $(u,v)$ of \eqref{E:1.5} is a ground sate solution
if $(u,v)$ is nontrivial and $\mathcal J(u,v)\leq \mathcal J(\phi,\psi)$ for any other nontrivial solution $(\phi,\psi)$ of \eqref{E:1.5}.

Our first result can be stated as follows:
\begin{theorem}\label{th:1.1}
Assume that $\beta, \lambda_i, \mu_i > 0 (i=1,2)$.  Then every positive solution $(u,v)\in X$ of \eqref{E:1.5}
is radially symmetric and monotone decreasing. In particular, $u$ and $v$ decrease exponentially.
\end{theorem}

We note that Wang and Shi \cite{WS} showed  the radial symmetry and 
the monotonic decreasing of positive solutions to  \eqref{E:1.4} for  the case $N=3$ and $\beta, \lambda_i, \mu_i > 0 (i=1,2)$.
Their approach relies on the moving plane method  of the integral form.
The methods of \cite{WS} can  apply to a general class of integral equations,
but they  can not be applied  to \eqref{E:1.5} since  $\Gamma_{2}(x) = -\frac{1}{2\pi}\ln(|x|)$ is sign--changing.
Inspired by \cite{CW}, we use a more direct and simpler variant of the moving plane method to prove Theorem \ref{th:1.1}.

To obtain a positive ground state solution of \eqref{E:1.5}, we define
\begin{align}\label{E:1.10}
\mathcal{N} : = \big\{(u,v)\in X \backslash \{(0,0)\} : N(u, v):=\big\langle \mathcal J^{\prime}(u,v),(u,v)\big\rangle=0 \big\},
\end{align}
\begin{align}\label{E:1.11}
c:=\inf\limits_{(u,v)\in\mathcal N} \mathcal J(u,v),
\end{align}
and
\begin{align}\label{E:1.12}
\beta_1 :=\frac{\mu_1 \big( \|\nabla u_1 \|_{2}^{2}+\lambda_{2} \|u_1\|_{2}^{2} \big) }{  \|\nabla u_1\|_{2}^{2}+\lambda_1 \|u_1\|_{2}^{2} },~~~ \beta_2 :=\frac{\mu_2 \big( \|\nabla u_2 \|_{2}^{2}+\lambda_{1}\|u_2\|_{2}^{2} \big) }{  \|\nabla u_2\|_{2}^{2}+\lambda_2 \|u_2\|_{2}^{2} },
\end{align}
where
$\lambda_i,\mu_i >0 (i=1,2)$ and $u_i$ is the unique ground sate solution of \eqref{E:1.6} with $(\lambda,\mu)=(\lambda_i,\mu_i)\ (i=1,2)$.

Our results on  the existence of positive ground state solutions to \eqref{E:1.5} are stated in the following Theorem.

\begin{theorem}\label{th:1.2}
Assume that $\beta, \lambda_i, \mu_i > 0 (i=1,2)$. If $\beta>\max{\{ \beta_{1},\beta_{2}\}}$,
then system \eqref{E:1.5} has a positive ground state solution $(u_0,v_0)$ in $X$,
where $\beta_{1}$ and $\beta_{2}$ are defined in \eqref{E:1.12}.
\end{theorem}

To obtain a ground sate solution of \eqref{E:1.5}, we need to overcome some difficulties. First, since the integral kernel $\Gamma_{2}(x) = -\frac{1}{2\pi}\ln(|x|)$ is sign--changing, we decompose $\ln(|x|)$ into $\ln(1+|x|)$ and $-\ln(1+\frac{1}{|x|})$.  Then, to make the corresponding functional sense, we introduce a new smaller working Space $X$, which is defined by \eqref{E:1.9}. Finally, by studying the constraint minimum  problem $c=\inf\limits_{(u,v)\in\mathcal N} \mathcal J(u,v)$ on the Nehari manifold restricted to $X$, we obtain the existence of positive ground state solutions of \eqref{E:1.5}.  Moreover, it is worth noticing that we also need to prove that $\mathcal N \in C^1$ is a natural constraint and eliminate the semi--trivial solutions in the processing of proving Theorem \ref{th:1.2}.

The paper is organized as follows.
The variational setting and preliminaries will be given in Section \ref{sc2}.
Section \ref{sc3} will devote to  the proof of Theorem \ref{th:1.2}.
In Section \ref{sc4}, we will complete the proof of Theorem \ref{th:1.1}.

\section{Variation framework and preliminaries}\label{sc2}
For convenience, we introduce  the following notations.

$\bullet$ \ A new  Hilbert space
$$X_1:=\bigg\{u \in H^{1}(\mathbb R^{2}):\int_{\mathbb R ^{2}} \ln(1+|x|)u^2 \mathrm{d}x < \infty \bigg\}$$
equipped with the norm
$$\displaystyle\|u\|_{X_1}^{2}:=\int_{\mathbb R^{2}} \big(|\nabla u|^{2}+u^{2}\big)\mathrm{d}x + \int_{\mathbb R^{2}}\ln(1+|x|)u^2 \mathrm{d}x;$$

$\bullet$ \ The standard norm in $L^{p}(\mathbb R^{2}) \ (1\leq p<\infty)$ is denoted  by $\displaystyle \|u\|_{p}:=\left( \int_{\mathbb R^2}|u|^{p} \mathrm{d}x \right)^{\frac{1}{p}}$;

$\bullet$ \ $L^{p}(\mathbb R^{2})\times L^{p}(\mathbb R^{2}) \ (1\leq p<\infty)$ denotes the Lebesgue space with the norm
$$\|(u,v)\|_{p}:=\left( \|u\|_p^p+ \|v\|_p^p\right)^{\frac{1}{p}} ;$$

$\bullet$ \ For any $u\in H^1(\mathbb R^2)$, $\displaystyle\|u\|_{\ast}^2:=\int_{ \mathbb R^{2}} \ln(1+|x|)u^2 \mathrm{d}x$;

$\bullet$ \ $C,\ C_{1},\ C_{2},\cdots$ stand for positive constants possibly different in different places.

We define the following symmetric bilinear forms
\begin{align}\label{E:2.1}
\aligned
(u,v)&\mapsto I_{1}(u,v):=\int_{\mathbb R^2}\int_{\mathbb R^2}\frac{1}{2\pi} \ln(1+{|x-y|}) u(x)v(y) \mathrm{d}x\mathrm{d}y ; \vspace{.4cm}\\
(u,v)&\mapsto I_{2}(u,v):=\int_{\mathbb R^2}\int_{\mathbb R^2}\frac{1}{2\pi} \ln(1+\frac{1}{|x-y|}) u(x)v(y) \mathrm{d}x\mathrm{d}y; \vspace{.4cm}\\
(u,v)&\mapsto I_{0}(u,v):=I_{1}(u,v)-I_{2}(u,v)=\int_{\mathbb R^2}\int_{\mathbb R^2}\frac{1}{2\pi}\ln(|x-y|)u(x)v(y)\mathrm{d}x\mathrm{d}y; \vspace{.4cm}\\
(u,v)&\mapsto B_{i}(u,v):=\mu_{1}I_{i}(u,u)+\mu_{2}I_{i}(v,v)+2\beta I_{i}(u,v),\ \ i=0,1,2.
\endaligned
\end{align}
Since $0\leq \ln(1+r)\leq r$ for $r \geq 0$,  by  Hardy-Littlewood-Sobolev inequality (See Theorem 4.3 of \cite{LB}), we have that
\begin{align}\label{E:2.2}
I_{2}(u,v) \leq \frac{1}{2\pi}  \int_{\mathbb R^2}\int_{\mathbb R^2} \frac{1}{|x-y|} u(x) v(y) \mathrm{d}x\mathrm{d}y
           \leq \mathcal C_{0} \|u\|_{\frac{4}{3}}\|v\|_{\frac{4}{3}},\ \ \forall u,v \in L^{\frac{4}{3}}(\mathbb R^{2}),
\end{align}
with a constant $\mathcal C_{0}>0$. Also, we  define the functionals:
\begin{align}\label{E:2.3}
\aligned
A_{1}&:H^{1}(\mathbb R^{2})\times H^{1}(\mathbb R^{2}) \to [0,\infty], \hspace{1.6cm}  A_{1}(u,v):=B_{1}(u^{2},v^{2}) ; \vspace{.4cm}\\
A_{2}&:L^{\frac{8}{3}}(\mathbb R^{2})\times L^{\frac{8}{3}}(\mathbb R^{2}) \to [0,\infty), \hspace{1.6cm}  A_{2}(u,v):=B_{2}(u^{2},v^{2}); \vspace{.4cm}\\
A_{0}&:H^{1}(\mathbb R^{2})\times H^{1}(\mathbb R^{2}) \to \mathbb R\cup \{\infty\}, \hspace{1.2cm}  A_{0}(u,v):=B_{0}(u^{2},v^{2}).\vspace{.4cm}\\
\endaligned
\end{align}
From \eqref{E:2.2}, we deduce that
\begin{align}\label{E:2.4}
A_{2}(u,v)&=\mu_1 I_2(u^2 , u^2)+ \mu_2 I_2(v^2 , v^2)+ 2 \beta I_2 (u^2 , v^2) \vspace{.1cm} \nonumber\\
          &\leq \mathcal C_{0}(\mu_{1}+\beta) \|u\|^{4}_{\frac{8}{3}}+\mathcal C_{0}(\mu_{2}+\beta)\|v\|^{4}_{\frac{8}{3}},
            \hspace{0.6cm}   \forall  (u,v)\in L^{\frac{8}{3}}(\mathbb R^{2})\times L^{\frac{8}{3}}(\mathbb R^{2}).
\end{align}

Since
$$\displaystyle \ln(1+|x-y|)\leq \ln(1+|x|+|y|)\leq\ln(1+|x|)+\ln(1+|y|), \ \ \forall x,y\in\mathbb R^{2},$$
one has
\begin{align}\label{E:2.5}
I_{1}(uv,\phi\psi)
&\leq \int_{\mathbb R^2}\int_{\mathbb R^2}\frac{1}{2\pi} \big(\ln(1+|x|)+\ln(1+|y|)\big) |u(x)v(x)| |\phi(y)\psi(y)| \mathrm{d}x\mathrm{d}y \vspace{.1cm} \nonumber\\
&\leq \frac{1}{2\pi} \big(\|u\|_{\ast} \|v\|_{\ast} \|\phi\|_{2} \|\psi\|_{2} + \|u\|_{2} \|v\|_{2} \|\phi\|_{\ast} \|\psi\|_{\ast} \big),
                  ~~~~ \forall (u,\phi), (v,\psi) \in X,
\end{align}
which implies  that, for any $(u,v) \in X$,
\begin{align}\label{E:2.6}
A_{1}(u,v)&=\mu_1 I_1(u^2 , u^2)+ \mu_2 I_1(v^2 , v^2)+ 2 \beta I_1 (u^2 , v^2) \vspace{.1cm} \nonumber\\
&\leq \frac{1}{ \pi} \left( \mu_1 \|u\|_{\ast}^{2}\|u\|_{2}^{2}+ \mu_2 \|v\|_{\ast}^{2}\|v\|_{2}^{2} + \beta   \|u\|_{\ast}^{2} \|v\|_{2}^{2}+\beta \|v\|_{\ast}^{2} \|u\|_{2}^{2} \right).
\end{align}

\begin{proposition}\label{pr:2.1} (Gagliardo-Nirenberg inequality)(See \cite{NL})
\ Let $u \in L^{q}(\mathbb R^{N})$ and it is derivatives of order $m$, $D^{m}u \in L^{r}(\mathbb R^{N})$, $1\leq q,r \leq \infty$.
For the derivatives $D^{j}u$, $0\leq j <m$, the following inequalities hold
\begin{align} \label{E:2.7}
\|D^{j}u\|_{p} \leq C \|u\|_{q}^{1-a} \|D^{m}u\|_{r}^{a},
\end{align}
where
$$\frac{1}{p}=\frac{j}{N}+a(\frac{1}{r}-\frac{m}{N})+(1-a)\frac{1}{q},$$
for all $\frac{j}{m} \leq a \leq 1$ and the constant $C$ depending only on $N,m,j,p,r,a$.
\end{proposition}

\begin{lemma}\label{lm:2.2} (Lemma 2.1 of \cite{CW})  \ Let $\{u_{n}\} $ be a sequence in $ L^{2}(\mathbb R^{2})$
such that $u_{n}\to u\in L^{2}(\mathbb R^{2})\setminus\{0\} $ pointwise a.e. on  $\mathbb R^2$
and $\{v_{n}\} $ be a bounded sequence in $L^{2}(\mathbb R^{2})$ such that
$$\sup\limits_{n\in \mathbb N} I_{1}(u_{n}^{2},v_{n}^{2})<\infty.$$
Then there exists $n_{0}\in \mathbb N$ and $C>0$ such that $\|v_{n}\|_{\ast}<C$ for $n\geq n_{0}$.
\end{lemma}

\begin{lemma}\label{lm:2.3}
We have the following properties:

$(i)$:  For all $p\in[2,\infty)$, the embedding  $X \hookrightarrow L^{p}(\mathbb R^{2})\times L^{p}(\mathbb R^{2})$ is compact;

$(ii)$: The functionals $A_{0},A_{1},A_{2}$ and $\mathcal J$ are of class $C^{1}$ on $X$. Moreover, for any $(u,v), (\phi,\psi) \in X $,
\begin{align} \label{E:2.8}
\langle A_{i}^{\prime}(u,v) ,(\phi,\psi) \rangle =&4\mu_1 I_{i}(u\phi,u^{2})+4\mu_2 I_{i}(v\psi,v^{2}) \vspace{.1cm}  \nonumber\\
&+ 4\beta I_{i}(v\psi,u^{2})+4\beta I_{i}(u\phi,v^{2}),~~~~ (i=0,1,2) ;
\end{align}

$(iii)$: $\mathcal J$ is weakly lower semicontinuous on $X$.
\end{lemma}

\begin{proof}
The proof is similar to that of Lemma 2.3 in \cite{CDZ} and Lemma 2.2 in \cite{CW}, so we omit it.
\end{proof}

\begin{lemma}\label{lm:2.4}
Let $u\in X_1 \setminus \{0\} $. Then $ w_u\in L_{loc}^{\infty}(\mathbb R^2)$, and
\begin{align}\label{E:2.9}
w_{u}(x)+\frac{1}{2\pi} \|u\|_{2}^{2} \ln(|x|) \to 0 \ \  \ \mbox{as}\ |x| \to  +\infty,
\end{align}
where $\displaystyle w_{u} (x):= - \int_{\mathbb R^{2}} \frac{1}{2\pi} \ln (|x-y|) u^{2}(y) \mathrm{d}y $.
Moreover, if $u \in  C_{loc}^{1,\alpha}(\mathbb R^{2})$ for any $0\leq\alpha< 1$,  then  $w_{u} $ is of class $C^3(\mathbb R^2)$
and  satisfies $-\triangle w_{u}=u^{2}$ in $\mathbb R^{2}$.
\end{lemma}
\begin{proof}
The proof has been given in  \cite{CDZ} or \cite{CW}.
\end{proof}

\begin{lemma}\label{lm:2.5}
Assume that $\beta , \lambda_i, \mu_i >0 ~ (i=1,2)$. If $(u,v)\in X \setminus \{(0,0)\}$ is a weak solution of \eqref{E:1.5}. Then

$(i)$:   $u,v\in C ^{2} (\mathbb R^{2})$;

$(ii)$:  $u, v$ decay exponentially, i.e., there exist $C_{1}>0$ and $C_{2}>0$ such that
\begin{align}\label{E:2.10}
|u(x)|,  |v(x)| \leq C_{1}e^{- C_{2}|x|}.
\end{align}
\end{lemma}

\begin{proof}
We set
$$f_{1}(x):=\mu_1 w_{u}(x)u(x)+\beta w_{v}(x)u(x),$$
and
$$f_{2}(x):=\mu_2 w_{v}(x)v(x)+\beta w_{u}(x)v(x),$$
where $w_{u}$ and $w_{v}$ are defined in Lemma \ref{lm:2.4}.
The system \eqref{E:1.5} can be written as
\begin{align}\label{E:2.11}
\left\{ \displaystyle\begin{array}{ll} \displaystyle
-\Delta u+ \lambda_1 u
=f_{1}(x),
&\mbox{in}\ \mathbb R^2, \vspace{.2cm}\\ \displaystyle
-\Delta v+ \lambda_2 v
=f_{2}(x) ,
&\mbox{in}\ \mathbb R^2.
\end{array}\right.
\end{align}
Then,  for any bounded open subset $W\subset\subset \mathbb R^{2}$,  we find that, for any $p\in[1,\infty)$,
\begin{align}\label{E:2.12}
\aligned
\|f_{1}(x)\|^p_{L^{p}(W)}&\leq C_1 \int_{W} | w^2_{u}(x)+2 u^2(x)+ w^2_{v}(x)|^{p} \mathrm{d}x \vspace{.2cm}\\
& \leq C_2 \big(\|w_{u}\|_{L^{\infty}(W)}^{2p} + \|u\|_{L^{2p}(W)}^{2p} +  \|w_{v}\|_{L^{\infty}(W)}^{2p}\big) \leq C, \vspace{.2cm}\\
\|f_{2}(x)\|^{p}_{L^{p}(W)}&\leq C_2 \big( \|w_{u}\|_{L^{\infty}(W)}^{2p} + \|v\|_{L^{2p}(W)}^{2p} +  \|w_{v}\|_{L^{\infty}(W)}^{2p}\big) \leq C,
\endaligned
\end{align}
since $ w_u,w_v \in L_{loc}^{\infty}(\mathbb R^2)$ and $(u,v)\in X\subset H$.
Due to  \eqref{E:2.12} and the arbitrariness of $W$, the Interior $H^{2}$--Regularity theory implies
that  $u,v \in W^{2,p}_{loc}(\mathbb R^{2})$ for any $p\in[1,\infty)$.
By Sobolve embedding, we have that $u,v\in  C_{loc}^{1,\alpha}(\mathbb R^{2})$ for any $0\leq\alpha <1$,
which, together with Lemma \ref{lm:2.4}, shows that  $w_{u},w_{v}\in C^3(\mathbb R^{2})$.
Then, it is easy to see that $f_{1},f_{2}$ are locally H\"{o}lder continuous.
Therefore, $u,v\in C^{2}(\mathbb R^{2})$ by elliptic regularity theorem.

It follows from  the Agmon's Theorem (See \cite{AS}) that  $u,v$ decay exponentially.  We complete the proof.
\end{proof}

\begin{lemma}\label{lm:2.6}
Assume that $\beta , \lambda_i, \mu_i >0 (i=1,2)$. Then

$(i)$:  For any $(u,v)\in X\setminus \{(0,0)\} $, there exists a positive constant $t_{0}>0$
such that
$$\big(t^2_0 u(t_0x),t_0^2 v(t_0x)\big)\in  \mathcal{N};$$

\vspace{.1cm}

$(ii)$:  There exists $\zeta>0$ such that $\|(u,v)\|_{H}\geq \zeta$ for any $(u,v)\in \mathcal{N}$;

\vspace{.1cm}

$(iii)$: $c=\inf\limits_{(u,v)\in\mathcal{N}} \mathcal{J}(u,v)>0.$
\end{lemma}

\begin{proof}  $(i)$:   For any $(u,v)\in X\setminus \{(0,0)\} $. Consider
\begin{align*}
g(t):&=N \big(t^{2}u(tx),t^{2}v(tx)\big)\\
& = t^{4} \big(\|\nabla u\|_2^2 +\|\nabla v\|_2^2 \big) +  t^{2}  \big(\lambda_1 \|u\|_2^2+ \lambda_2 \|v\|^2_2\big)  \vspace{.1cm} \\
&\quad\,+ t^{4}  A_{0}(u,v)-\frac{t^{4}\ln t}{2\pi}\big( \mu_{1} \|u\|^{4}_{2} + \mu_{2} \|v\|^{4}_{2} + 2\beta \|u\|^{2}_{2} \|v\|^{2}_{2} \big), ~~~ t>0.
\end{align*}
It is easy to  see that $\lim\limits_{t\to 0^+} g(t) =0^+$  and $\lim\limits_{t\to +\infty }g(t) = - \infty$.
So there exists a   positive constant $t_0>0$ such that $N \big(t^2_0 u(t_0x),t_0^2 v(t_0x)\big) =0$, i.e. $\big(t^2_0 u(t_0x),t_0^2 v(t_0x)\big)\in  \mathcal{N}$.

\ $(ii)$:  We have
\begin{align*}
\|(u,v)\|_{H}^2
&=-A_0(u,v)=A_2 (u,v)-A_1 (u,v) \\
&\leq  A_2 (u,v)  \leq C_1 ( \|u\|^{4}_{\frac{8}{3}}+\|v\|^{4}_{\frac{8}{3}} )  \leq C \|(u,v)\|_{H}^4,
\ \ \forall (u,v)\in \mathcal N,
\end{align*}
which implies that, for any $(u,v)\in \mathcal N$,
\begin{align}\label{E:2.13}
\|(u,v)\|_{H}^{2}\geq \frac{1}{C}=:\zeta^2>0.
\end{align}

$(iii)$: Using $(ii)$,  we get that, for any $(u,v)\in \mathcal N$,
\begin{align*}
\mathcal J(u,v)&=\mathcal J(u,v)-\frac{1}{4} N(u,v)
= \frac{1}{4} \|(u,v)\|_H^2\geq \frac{1}{4} \zeta^2  >0.
\end{align*}
Hence $c=\inf\limits_{(u,v)\in\mathcal N} \mathcal{J}(u,v)\geq \frac{1}{4} \zeta^2  >0$.
\end{proof}

\begin{lemma}\label{lm:2.7}   Assume that $\beta , \lambda_i, \mu_i >0 (i=1,2)$. Then $\mathcal N$ is a $C^{1}$--manifold
and any critical point of $\mathcal{J}|_\mathcal{N}$ is a critical point of $\mathcal{J}$ in $X$.
\end{lemma}

\begin{proof} \ Following from Lemma \ref{lm:2.6}$(i)$, we have that $\mathcal{N} \neq \emptyset$.
Now, we divided our proof into two steps.

$(i)$:  By \eqref{E:2.4} and the Sobolev embedding inequality, we  find that, for any $r>0$ small enough,
\begin{align*}
N (u,v)
&=\|(u,v)\|^2_H + A_0 (u,v)\vspace{.1cm} \geq \|(u,v)\|^{2}_{H}-(\mu_{1}+\beta)\mathcal C_{0} \|u\|^{4}_{\frac{8}{3}}
-(\mu_{2}+\beta)\mathcal C_{0} \|v\|^{4}_{\frac{8}{3}} \vspace{.1cm} \\
&\geq \|(u,v)\|_{H}^{2}-C \|(u,v)\|_{H}^{4}>0, \ \ \forall \ \|(u,v)\|_{H}=r,
\end{align*}
which means that $(0,0)\notin \partial\mathcal{N}$.

On the other hand, we can obtain  that, for any $(u,v) \in \mathcal{N}$,
\begin{align}\label{E:2.14}
\langle N^{\prime}(u,v),(u,v)\rangle = 2\|(u,v)\|^2_H + 4 A_0 (u,v) = -2\|(u,v)\|^2_H< 0,
\end{align}
which, together with the  Implicit Function Theorem, implies that $\mathcal{N}$ is a $C^{1}$--manifold.

$(ii)$:  If $(u,v)$ is a critical point of $\mathcal{J}|_\mathcal{N}$, i.e. $(u,v)\in \mathcal{N}$ and $(\mathcal{J}|_\mathcal{N})^{\prime}(u,v)=(0,0)$.
Then there is a Lagrange multiplier $\gamma \in \mathbb R$ such that
\begin{align}\label{E:2.15}
\mathcal J^{\prime} (u,v) - \gamma  N^{\prime}(u,v)=(0,0).
\end{align}
Testing \eqref{E:2.15} with $(u,v)$, we get that
\begin{align}\label{E:2.16}
0=\langle \mathcal J^{\prime}(u,v),(u,v)\rangle - \gamma \langle N^{\prime}(u,v),(u,v)\rangle=-\gamma \langle N^{\prime}(u,v),(u,v)\rangle.
\end{align}
From \eqref{E:2.14} and \eqref{E:2.16}, we  get that  $\gamma = 0 $. Hence, $\mathcal J^{\prime} (u,v)=(0,0)$,
i.e. $(u,v)$ is  a critical point of $\mathcal{J}$ in $X$.
The proof of Lemma \ref{lm:2.7} is completed.
\end{proof}

\begin{lemma}\label{lm:2.8}  \ Assume that $\beta,\lambda_i, \mu_i >0 (i=1,2)$. If $\beta>\max\{\beta_1 ,\beta_2\}$, then we have
$$c<\min\{\mathcal J(u_{1},0),\mathcal J(0,u_{2})\},$$
where $\beta_1$ and $\beta_2$ are defined in \eqref{E:1.12},
and $u_{i}$ is the unique ground state solution of \eqref{E:1.6} with $(\lambda, \mu)=(\lambda_{i}, \mu_{i}) (i=1,2)$.
\end{lemma}

\begin{proof} \ Without loss of generality, we may assume that $\mathcal J(u_1,0)\leq \mathcal J(0,u_2)$. For any $\rho\geq 0$, let
\begin{align}\label{E:2.17}
F_{\rho}(t):=N(tu_{1},t \rho u_{1})=t^{2}\|(u_1,\rho u_1)\|_{H}^2 +t^{4}( \mu_1 +\rho^4\mu_2 +2\rho^2\beta )I_0 (u_1^2,u_1^2).
\end{align}
Since $u_1$ is the unique ground state solution of \eqref{E:1.6} with $(\lambda, \mu)=(\lambda_{1}, \mu_{1})$, one has that
\begin{align}\label{E:2.18}
\|\nabla u_{1}\|_{2}^{2}+\lambda_{1}\|u_{1}\|_{2}^{2}+\mu_{1}I_0 (u_1^2,u_1^2)&=0,
\end{align}
which means that $I_0 (u_1^2,u_1^2)<0$. So  we can see that
\begin{align}\label{E:2.19}
t_{\rho}=\left(\frac{\|(u_1,\rho u_1)\|_{H}^2}{(\mu_1 +\rho^4\mu_2 +2\rho^2\beta )\left(-I_0 (u_1^2,u_1^2)\right)}\right)^\frac{1}{2}
\end{align}
is the unique positive root of $F_{\rho}(t)=0$, which implies that $(t_{\rho} u_1 ,t_{\rho} \rho u_1) \in \mathcal N$ for any $\rho\geq 0$. By \eqref{E:2.18} and \eqref{E:2.19}, we can find that, for any $\rho\geq 0$,
\begin{align*}
h(\rho):= \mathcal J(t_{\rho} u_1,t_{\rho} \rho u_1)
&=\frac{t_{\rho}^2}{2}\|(u_1,\rho u_1)\|_{H}^2+\frac{t_{\rho}^4}{4}\big(\mu_1 +\rho^4 \mu_2 +2\rho^2 \beta \big)I_0 (u_1^2,u_1^2)  \vspace{.1cm}\\
&= \frac{\|(u_1,\rho u_1)\|_{H}^4}{4(\mu_1 +\rho^4 \mu_2 +2\rho^2 \beta)\left(-I_0 (u_1^2,u_1^2)\right)}\\
&= \frac{\mu_1 \bigg(\|\nabla u_1 \|_2^2+\lambda_1 \|u_1\|_2^2+\rho^2 \big(\|\nabla u_1 \|_2^2+\lambda_2\|u_1\|_2^2 \big)\bigg)^2  }{4(\mu_1 +\rho^4 \mu_2 +2\rho^2 \beta)\big(\|\nabla u_1 \|_2^2 +\lambda_1 \|u_{1}\|_2^2\big)}\\
&= \frac{\mu_1 (b_1+\rho^2 b_2)^2  }{4(\mu_1 +\rho^4 \mu_2 +2\rho^2 \beta )b_1},
\end{align*}
where $b_i : =\|\nabla u_1 \|_2^2 +\lambda_i \|u_{1}\|_2^2 \ (i=1,2)$.
Since $\beta>\beta_1$, we have $ b_2 \mu_1 - b_1 \beta <0$, which, together with the derivative of $h(\rho)$,  gives that
\begin{align}\label{E:2.20}
h^{\prime}(\rho)&=\frac{\mu_1 \bigg[ 4\big(b_1 + \rho^2 b_2  \big) \rho b_2  \big(\mu_1 +\rho^4 \mu_2 +2\rho^2 \beta \big) - \big(4 \rho^3 \mu_2 + 4 \rho \beta \big)\big(b_1 + \rho^2 b_2 \big)^2 \bigg] }{4 b_1 \big(\mu_1 +\rho^4 \mu_2 +2\rho^2 \beta \big)^2} \nonumber\\
&=\frac{\mu_1 \bigg[ \big( 4b_1 b_2 \mu_1 - 4 b_1^2 \beta  \big) \rho + o(\rho^2) \bigg] }{4 b_1 \big(\mu_1 +\rho^4 \mu_2 +2\rho^2 \beta \big)^2} \to 0^-, ~~~~~~~~~~ \mbox{as}~\rho \to 0^+.
\end{align}
So there exists $\rho_1>0$ such that
\begin{align}\label{E:2.21}
c\leq \mathcal J(t_{\rho_1} u_1 ,t_{\rho_1} \rho _1 u_1)=h(\rho_1)
&<h(0)=\frac{1}{4} b_1 =\frac{1}{4} \big( \|\nabla u_{1}\|_{2}^{2}+\lambda_{1}\|u_{1}\|_{2}^{2} \big) \nonumber\\
&=\frac{1}{2} \big( \|\nabla u_{1}\|_{2}^{2}+\lambda_{1}\|u_{1}\|_{2}^{2} \big) + \frac{1}{4}\mu_{1}I_0 (u_1^2,u_1^2) \nonumber\\
&=\mathcal J(u_{1},0)=\min\{\mathcal J(u_{1},0),\mathcal J(0,u_{2})\}.
\end{align}
We complete the proof.
\end{proof}

\section{The  Proof of Theorem \ref{th:1.2}}\label{sc3}

\begin{proof}[\textbf{Proof of Theorem \ref{th:1.2}:}]  Assume that   $\{(u_{n},v_{n})\}\subset \mathcal N$ such that
$\mathcal J(u_n,v_n)\to c$ as $n\to\infty$.
It is easy to see that
\begin{align*}
\aligned
c+o(1)&=\mathcal J(u_{n},v_{n})-\frac{1}{4} N(u_n,v_n)  = \frac{1}{4} \|(u_n,v_n)\|_{H}^2,
\endaligned
\end{align*}
which tells us  that  $\{(u_{n},v_{n})\}$ is bounded in $H$.
Using \eqref{E:2.7}, we have $ \|u\|_{\frac{8}{3}}^{4} \leq C \|u\|_{2}^{3}\|\nabla u\|_{2}$ for any $u\in H^1(\mathbb R^2)$, which, together with  \eqref{E:2.4}, means that
\begin{align}\label{E:3.1}
\aligned
0 \leq A_{2}(u_n,v_n) &\leq \mathcal C_{0}(\mu_{1}+\beta) \|u_n\|^{4}_{\frac{8}{3}}+\mathcal C_{0}(\mu_{2}+\beta)\|v_n\|^{4}_{\frac{8}{3}} \vspace{.1cm}\\
& \leq C_1 \big( \|u_n\|_{2}^{3}\|\nabla u_n\|_{2}+ \|v_n\|_{2}^{3}\|\nabla v_n\|_{2} \big) \vspace{.1cm}\\
&\leq C.
\endaligned
\end{align}
So we can get that
\begin{align*}
c+o(1)= \mathcal J (u_n, v_n)&=\frac{1}{2} \|(u_n, v_n)\|_{H}^2 + \frac{1}{4} A_0(u_n, v_n) \\
&\geq \frac{1}{4} A_1 (u_n, v_n)-\frac{1}{4} A_2 (u_n, v_n) \\
&\geq \frac{1}{4} \big( \mu_1I_{1}(u^2_n,u^2_n)+ \mu_2I_{1}(v^2_n,v^2_n)+ 2\beta I_{1}(u^2_n,v^2_n) \big)-C ,
\end{align*}
which implies that $\{I_{1}(u^2_n,v^2_n)\}$ is bounded. Since $\{I_{1}(u^2_n,v^2_n)\}$ and $\{\|(u_{n},v_{n})\|_H^2\}$ are bounded  in $\mathbb R$,
it follows from Lemma \ref{lm:2.2}  that $\{ \|u_n\|_{\ast}^2 + \|v_n\|_{\ast}^2 \}$ is bounded. So $\{(u_{n},v_{n})\} $ is bounded in $X$.
Passing to a subsequence, one has that
\begin{align} \label{E:3.2}
\aligned
&(u_{n} , v_{n}) \rightharpoonup (u_0  , v_0 )\ \ \mbox{in}\ X  ,  \vspace{.1cm}\\
&(u_{n} , v_{n}) \to (u_0  , v_0 )  \ \ \mbox{in}\ L^{p}(\mathbb R^{2})\times L^{p}(\mathbb R^{2}) \ \ \mbox{for} \ p \in [2,\infty),  \vspace{.1cm}\\
&(u_{n} , v_{n}) \to (u_0  , v_0 )  \ \ \mbox{a.e.} \ \ \mbox{on}\ \ \mathbb R^{2}.
\endaligned
\end{align}
By the definition of $A_0, A_1, A_2$  and \eqref{E:3.1}, we find
\begin{align*}
\aligned
c+o(1)&= \mathcal J (u_{n},v_{n}) -\frac{1}{2} N(u_n,v_n) =-\frac{1}{4} A_0(u_n,v_n) \vspace{.1cm}\\
&\leq \frac{1}{4} A_2(u_n,v_n) \leq C_1 \big( \|u_{n}\|_{2}^{3}\|\nabla u_{n}\|_{2}+ \|v_{n}\|_{2}^{3}\|\nabla v_{n}\|_{2} \big) \\
&\leq C (\|u_{n}\|_{2}+\|v_{n}\|_{2})^3 (\|\nabla u_{n}\|_{2}+\|\nabla v_{n}\|_{2})\\
&\leq C(\|u_{n}\|_{2}+\|v_{n}\|_{2})^3,
\endaligned
\end{align*}
which shows that $\|u_{n}\|_{2}+\|v_{n}\|_{2} > 0$.  Using \eqref{E:3.2}, we get that $(u_0 , v_0 )\neq (0,0)$.
Passing to $(|u_n|,|v_n|)$, we may assume that $(u_0,v_0)$ is nonnegative.

By the weak lower semicontinuity of norm and $\mathcal J$ in $X$,
we can conclude that $N(u_0, v_0) \leq \liminf\limits_{n\to \infty} N(u_n , v_n) =0$,
which, together with $\lim\limits_{t\to 0^+}N(t^2 u_0(t x) , t^2 v_0(t x))=0^+$, implies that  there exists $0<t_0 \leq 1$
such that $(t_0^2 u_0(t_0 x) , t_0^2 v_0(t_0 x))\in \mathcal N$.
Therefore,  we find that
\begin{align*}
c&=\lim_{n\to \infty}\mathcal J(u_n , v_n )\\
&=\lim_{n\to \infty}\big[ \mathcal J(u_n , v_n )-\frac{1}{4} N(u_n , v_n) \big]  \vspace{.1cm}\\
&=\lim_{n\to \infty} \frac{1}{4} \|(u_n , v_n)\|_H^2 \\
&\geq \frac{1}{4} \|(u_0 , v_0 )\|_H^2 \\
&\geq \frac{t_0^4}{4} (\|\nabla u_0\|_2^2+\|\nabla v_0\|_2^2)   +  \frac{t_0^2}{4} (\lambda_1 \|u_0\|_2^2+\lambda_2 \|v_0\|_2^2) \\
&=\frac{1}{4} \|(t_0^2 u_0(t_0 x) , t_0^2 v_0(t_0 x) )\|_H^2 \\
&=\mathcal J(t_0^2 u_0(t_0 x) , t_0^2 v_0(t_0 x) ) \geq c,
\end{align*}
which gives that $t_0=1$. So $\mathcal J(u_0, v_0)=c$ and $(u_0, v_0) \in \mathcal N $, i.e. $c$ is achieved by $(u_0,v_0)$.
 From Lemma \ref{lm:2.8}, we get that
$$J(u_0 , v_0)=c<\min{\{ \mathcal J(u_{1},0), \mathcal J(0,u_{2})\}}.$$
Using  Lemmas \ref{lm:2.5} and \ref{lm:2.7}, we find that    $(u_0 , v_0)\in   C ^{2} (\mathbb R^{2})\times  C ^{2} (\mathbb R^{2})$
is a nonnegative nontrivial ground state solution of \eqref{E:1.5}.
Combining   the following decomposition
$$\int_{\mathbb R^{2}} \ln (|x-y|) u^{2}(y) \mathrm{d}y =\int_{\mathbb R^{2}} \ln(1+|x-y|) u^{2}(y) \mathrm{d}y - \int_{\mathbb R^{2}} \ln(1+\frac{1}{|x-y|}) u^{2}(y) \mathrm{d}y, \ \ \forall u \in X_1$$
and the strong maximum principle, we  have that $(u_0 , v_0)$ is  a positive ground state solution of \eqref{E:1.5}.
The proof of Theorem \ref{th:1.2} is completed.
\end{proof}

\section{The Proof of Theorem \ref{th:1.1}}\label{sc4}
In this section, for any $t\in\mathbb R$, we let
$$ H_t :=\{ x=(x_1, x_2)\in\mathbb R^2 : x_1 > t \},$$
$$ T_t := \{ x=(x_1, x_2)\in\mathbb R^2 : x_1 = t \},$$
and
$$x^t:=(2t-x_1,x_2),~~~\forall x=(x_1,x_2)\in\mathbb R^2.$$
Assume that $(u,v)$ is  a fixed positive solution of \eqref{E:1.5} in $X$  and  set
$$u^t(x):=u(x^t), \ v^t(x):=v(x^t), \ w_u^t(x):=w_u(x^t), \ w_v^t(x):=w_v(x^t), \ \  \ x\in\mathbb R^2,\ t \in \mathbb R, $$
and
$$u_t:=u^t-u, \ v_t:=v^t-v, \ w_{u,t}:= w_u^t-w_u , \  w_{v,t}:= w_v^t-w_v \ \mbox{in}\ H_t,$$
where $w_{u},w_{v}$ are defined in Lemma \ref{lm:2.4}.
By direct computation, we have that
\begin{align}\label{E:4.1}
\left\{ \displaystyle\begin{array}{ll} \displaystyle
-\Delta u_t + \lambda_1 u_t
=\mu_1 w_{u,t}u^t+ \mu_1 w_{u}u_t+\beta w_{v,t}u^t+\beta w_{v}u_t,
&\mbox{in}\ H_t, \vspace{.2cm}\\ \displaystyle
-\Delta v_t+ \lambda_2 v_t
=\mu_2 w_{v,t}v^t+ \mu_2 w_{v}v_t+\beta w_{u,t}v^t+\beta w_{u}v_t,
&\mbox{in}\ H_t, \vspace{.2cm}\\ \displaystyle
-\Delta w_{u,t} = (u^t)^2- u^2 =(u^t+u)u_t,
&\mbox{in}\ H_t, \vspace{.2cm}\\ \displaystyle
-\Delta w_{v,t} =(v^t)^2- v^2  =(v^t+v)v_t,
&\mbox{in}\ H_t.
\end{array}\right.
\end{align}

\begin{lemma}\label{lm:4.1}
Assume that  $(u, v)$ is a positive solution of \eqref{E:1.5} and $w_{u},w_{v}$ are defined in Lemma \ref{lm:2.4}. Then we have,
\begin{align}\label{E:4.2}
\aligned
w_{u,t}(x)=w_u^t-w_u=\int_{H_t} \frac{1}{2\pi} \ln\left(\frac{|x-y^t|}{|x-y|}\right) \left( u^t(y)+u(y) \right)u_t(y)  \mathrm{d}y, ~~~ x\in\mathbb R^2,  \\
w_{v,t}(x)=w_v^t-w_v=\int_{H_t} \frac{1}{2\pi} \ln\left(\frac{|x-y^t|}{|x-y|}\right) \left( v^t(y)+v(y) \right)v_t(y)  \mathrm{d}y, ~~~ x\in\mathbb R^2 .
\endaligned
\end{align}
\end{lemma}
\begin{proof}
Since $|x^t-y^t|=|x-y|$ and $|x^t-y|=|x-y^t|$, we find that
\begin{align*}
w_{u,t}(x)&=w_u^t-w_u=-\int_{\mathbb R^2}\frac{1}{2\pi}\ln (|x^t-y|) u^2(y) \mathrm{d}y  +  \int_{\mathbb R^2}\frac{1}{2\pi}\ln (|x-y|) u^2(y) \mathrm{d}y \nonumber\\
&=      -\int_{H_t}\frac{1}{2\pi}\ln (|x^t-y|) u^2(y) \mathrm{d}y -  \int_{\mathbb R^2 \backslash H_t }\frac{1}{2\pi}\ln (|x^t-y|) u^2(y) \mathrm{d}y \nonumber\\
&\quad\, +  \int_{H_t}\frac{1}{2\pi}\ln (|x-y|) u^2(y) \mathrm{d}y +  \int_{\mathbb R^2 \backslash H_t }\frac{1}{2\pi}\ln (|x-y|) u^2(y)    \mathrm{d}y \nonumber\\
&=      -\int_{H_t}\frac{1}{2\pi}\ln (|x-y^t|) u^2(y) \mathrm{d}y -  \int_{ H_t }\frac{1}{2\pi}\ln (|x^t-y^t|) u^2(y^t) \mathrm{d}y \nonumber\\
&\quad\, +  \int_{H_t}\frac{1}{2\pi}\ln (|x-y|) u^2(y) \mathrm{d}y +  \int_{ H_t }\frac{1}{2\pi}\ln (|x-y^t|) u^2(y^t)    \mathrm{d}y \nonumber\\
&=\int_{H_t}\frac{1}{2\pi}\left( \ln(|x-y^t|)-\ln(|x-y|) \right) \left(u^2(y^t)-u^2(y) \right) \mathrm{d}y \nonumber\\
&=\int_{H_t}\frac{1}{2\pi}\ln \left(\frac{|x-y^t|}{|x-y|}\right) \left( u^t(y)+u(y) \right)u_t(y)  \mathrm{d}y.
\end{align*}
Similarly, one has that
$$\displaystyle w_{v,t}(x)=\int_{H_t} \frac{1}{2\pi} \ln\left(\frac{|x-y^t|}{|x-y|}\right) \left( v^t(y)+v(y) \right)v_t(y)  \mathrm{d}y.$$
We complete the proof.
\end{proof}

\begin{lemma}\label{lm:4.2}(Lemma 6.2 of \cite{CW})
There exists a constant $k>0$ such that
\begin{align}\label{E:4.3}
\|w_{u,t}^-\|_{L^2(H_t)} \leq k c_{u,t} \|u_t^-\|_{L^2(H_t)}, ~~~ \mbox{for}~ \mbox{every} ~ t\in\mathbb R,
\end{align}
where
\begin{align}\label{E:4.4}
c_{u,t}= \left( \int_{H_t^u} (y_1-t)^2 u^2(y) \mathrm{d}y \right)^\frac{1}{2},   ~~~  H_t^u:=\{ x\in H_t : u_t(x)<0 \},
\end{align}
and  $h^-:=\min \{ h,0 \}$ for any $h\in X_1$.
\end{lemma}

\begin{lemma}\label{lm:4.3}
Assume that $\beta,\lambda_i,\mu_i > 0(i=1,2)$. There exists $T>0$ such that, when $t\geq T$,    $u_t , v_t \geq 0$ in $H_t$.
\end{lemma}

\begin{proof}
From \eqref{E:2.9}, we can  choose $T_1>0$ such that $w_u, w_v \leq 0$ in $H_{t}$ for every $t>T_1$. By \eqref{E:4.1} and \eqref{E:4.3}, we have
\begin{align}\label{E:4.5}
\|u_t^-\|^2_{L^2(H_t)}&+\|v_t^-\|^2_{L^2(H_t)} \leq \|(u_t^-, v_t^-)\|^2_{H^1(H_t)\times H^1(H_t)} \nonumber\\
&=\int_{H_t}\left( \mu_1 w_{u,t}u^t u_t^- + \mu_1 w_{u}(u_t^-)^2 +\beta w_{v,t}u^t u_t^- + \beta w_{v}(u_t^-)^2 \right) \mathrm{d}x  \nonumber\\
&\quad +\int_{H_t}\left( \mu_2 w_{v,t}v^t v_t^- + \mu_2 w_{v}(v_t^-)^2 +\beta w_{u,t}v^t v_t^- + \beta w_{u}(v_t^-)^2 \right) \mathrm{d}x  \nonumber\\
&\leq \int_{H_t}\left( \mu_1 w_{u,t}^- u^t u_t^-  +\beta w_{v,t}^- u^t u_t^- + \mu_2 w_{v,t}^- v^t v_t^- +\beta w_{u,t}^- v^t v_t^-  \right) \mathrm{d}x \nonumber\\
&\leq C_1    \|w_{u,t}^-\|_{L^2(H_t)} \bigg( \|u^t\|_{L^{\infty}(H_t)}\|u_t^-\|_{L^2(H_t)} +\|v^t\|_{L^{\infty}(H_t)}\|v_t^-\|_{L^2(H_t)} \bigg) \nonumber\\
&\quad + C_1 \|w_{v,t}^-\|_{L^2(H_t)}  \bigg( \|u^t\|_{L^{\infty}(H_t)}\|u_t^-\|_{L^2(H_t)} +\|v^t\|_{L^{\infty}(H_t)}\|v_t^-\|_{L^2(H_t)} \bigg) \nonumber\\
&\leq  C_2 \bigg( c_{u,t} \|u_t^-\|_{L^2(H_t)}+ c_{v,t} \|v_t^-\|_{L^2(H_t)}\bigg) \bigg(  \|u_t^-\|_{L^2(H_t)} + \|v_t^-\|_{L^2(H_t)} \bigg) \nonumber\\
&\leq  C   \left( c_{u,t} + c_{v,t}  \right) \left( \|u_t^-\|^2_{L^2(H_t)} +\|v_t^-\|^2_{L^2(H_t)} \right).
\end{align}
By the definitions of $ c_{u,t}, c_{v,t}$ and the fact that  $u,v$ decay exponentially,  we find that
$$ \lim_{t\to \infty }(c_{u,t} +c_{v,t}) =0,$$
which, together with  \eqref{E:4.5}, implies that
 there exists $T>T_1$ such that, for any $t\geq T$,
$$u_t^- \equiv 0,\ v_t^- \equiv 0, ~~x\in H_t.$$
We complete the proof.
\end{proof}

\begin{lemma}\label{lm:4.4}
Assume that $\beta,\lambda_i,\mu_i > 0(i=1,2)$. Let $t\in\mathbb R$ and $u_t \geq 0, v_t \geq 0$ in $H_t$. Then

$(i)$:  $w_{u,t}\geq 0, w_{v,t}\geq 0$ in $H_t$;

$(ii)$:   If $u_t \not\equiv 0 $ or $v_t \not\equiv 0$, then we have that
\begin{align}\label{E:4.6}
u_t>0,\ v_t>0 \ \mbox{in} \ H_t  \ \  \mbox{and} \ \ \frac{\partial u }{\partial x_1 } < 0,\ \frac{\partial v }{\partial x_1 } < 0 \ \mbox{in} \ T_t.
\end{align}


\end{lemma}

\begin{proof} \
$(i)$: Since $\ln\left(\frac{|x-y^t|}{|x-y|}\right)>0$ for every $x,y\in H_t$, by Lemma \ref{lm:4.1}, we find that, for every $t\in\mathbb R$,
\begin{align}\label{E:4.7}
w_{u,t}, \ w_{v,t} \geq 0 \ \ \mbox{in} \ H_t.
\end{align}

$(ii)$: Without loss of generality, we  assume that $u_t \not\equiv 0$  in $H_t$.  Then  we get that $ w_{u,t} > 0$ in $H_t$  by \eqref{E:4.2}. So,  \eqref{E:4.1} and \eqref{E:4.7}
imply  that
$$-\Delta u_t + \left(\lambda_1 - \mu_1 w_u^- - \beta w_v^- \right) u_t
=\left( \mu_1 w_{u,t} +\beta w_{v,t} \right) u^t+ \left( \mu_1 w_u^+ + \beta w_v^+ \right) u_t >0, \ \ \
 \mbox{in}\ H_t, $$
and
$$-\Delta v_t+ \left( \lambda_2 - \mu_2 w_v^- - \beta w_u^- \right) v_t
=\left( \mu_2 w_{v,t} +\beta w_{u,t} \right) v^t +\left( \mu_2 w_v^+ + \beta w_u^+ \right) v_t> 0 ,\ \ \
 \mbox{in}\ H_t,$$
where  $h^+:=\max \{ h,0 \}$ for any $h\in X$. Hence $u_t>0,\ v_t>0$ in $H_t$ by the maximum principle, and
$$-2\frac{\partial u }{\partial x_1 }=\frac{\partial u_t }{\partial x_1 } > 0,\ \ \ -2\frac{\partial v }{\partial x_1 }=\frac{\partial v_t }{\partial x_1 } > 0 \ \ \mbox{in} \ T_t ,$$
by the Hopf Lemma.

We complete the proof.
\end{proof}

\begin{lemma}\label{lm:4.5}
Assume that $\beta,\lambda_i,\mu_i > 0(i=1,2)$.
Let $u_{t}(x), v_{t}(x) \geq 0$ in $H_{t}$, but $u_{t}\not\equiv0$ or $v_{t}\not\equiv0$ in $H_{t}$.
Then there exists $\varepsilon > 0$ such that, for any $\tau\in(t-\varepsilon, t]$, $u_\tau \geq 0, v_\tau \geq 0$ in $H_\tau$.
\end{lemma}

\begin{proof}
Let $B_R:=B_R (0) $ for $R>0$.
By \eqref{E:2.9} and  the fact that $u,v$ decay exponentially, we may choose $R>1$ large enough such that, for every $\tau\in\mathbb R$,
\begin{align}\label{E:4.8}
w_u \leq 0,\ w_v \leq 0~~~ \mbox{in}~ H_\tau \backslash B_R,
\end{align}
and  for any $\tau\in[t-1,t]$,
\begin{align}\label{E:4.9}
\left( \int_{\mathbb R^2\backslash B_R} (y_1-\tau)^2 u^2(y) \mathrm{d}y \right)^\frac{1}{2} + \left( \int_{\mathbb R^2\backslash B_R} (y_1-\tau)^2 v^2(y) \mathrm{d}y \right)^\frac{1}{2} < \frac{1}{2 C_0},
\end{align}
where $C_0 : = k (\mu_1+\mu_2+2\beta)(\|u^\tau\|_{L^{\infty}(H_\tau)}+\|v^\tau\|_{L^{\infty}(H_\tau)})$ and $k$ is defined    in Lemma \ref{lm:4.2}.
By Lemma \ref{lm:4.4}$(ii)$ and $u,v\in  C ^{2} (\mathbb R^{2})$,  there exists $0< \varepsilon <1 $ such that, for any $\tau\in(t-\varepsilon, t]$,
\begin{align}\label{E:4.10}
u_\tau > 0, \ v_\tau >0 \ \  \mbox{in} \ H_\tau \cap B_R,
\end{align}
which means that
\begin{align}\label{E:4.11}
u^-_\tau = 0, \ v^-_\tau =0 \ \  \mbox{in} \ H_\tau \cap B_R.
\end{align}
Using  \eqref{E:4.1}, \eqref{E:4.3}, \eqref{E:4.8} and \eqref{E:4.11}, we can get that
\begin{align}\label{E:4.12}
\|u_\tau^-\|^2_{L^2(H_\tau)}&+\|v_\tau^-\|^2_{L^2(H_\tau)} \leq \|(u_\tau^-, v_\tau^-)\|^2_{H^1(H_\tau)\times H^1(H_\tau)} \nonumber\\
&=\int_{H_\tau \backslash B_R}\left( \mu_1 w_{u,\tau}u^\tau u_\tau^- + \mu_1 w_{u}(u_\tau^-)^2 +\beta w_{v,\tau}u^\tau u_\tau^- + \beta w_{v}(u_\tau^-)^2 \right) \mathrm{d}x \nonumber\\
&\quad + \int_{H_\tau \backslash B_R}\left( \mu_2 w_{v,\tau}v^\tau v_\tau^- + \mu_2 w_{v}(v_\tau^-)^2 +\beta w_{u,\tau}v^\tau v_\tau^- + \beta w_{u}(v_\tau^-)^2 \right) \mathrm{d}x \nonumber\\
&\leq \int_{H_\tau \backslash B_R}\left( \mu_1 w_{u,\tau}^- u^\tau u_\tau^-  +\beta w_{v,\tau}^- u^\tau u_\tau^- + \mu_2 w_{v,\tau}^- v^\tau v_\tau^- +\beta w_{u,\tau}^- v^\tau v_\tau^-  \right) \mathrm{d}x \nonumber\\
&\leq  \|w_{u,\tau}^-\|_{L^2(H_\tau)} \bigg( \mu_1 \|u^\tau\|_{L^{\infty}(H_\tau)}\|u_\tau^-\|_{L^2(H_\tau)} +\beta \|v^\tau\|_{L^{\infty}(H_\tau)}\|v_\tau^-\|_{L^2(H_\tau)} \bigg) \nonumber\\
&\quad +  \|w_{v,\tau}^-\|_{L^2(H_\tau)}  \bigg( \beta \|u^\tau\|_{L^{\infty}(H_\tau)}\|u_\tau^-\|_{L^2(H_\tau)} +\mu_2 \|v^\tau\|_{L^{\infty}(H_\tau)}\|v_\tau^-\|_{L^2(H_\tau)} \bigg) \nonumber\\
&\leq    k c_{u,\tau} \|u_\tau^-\|_{L^2(H_\tau)} \bigg( \mu_1 \|u^\tau\|_{L^{\infty}(H_\tau)}\|u_\tau^-\|_{L^2(H_\tau)} + \beta \|v^\tau\|_{L^{\infty}(H_\tau)}\|v_\tau^-\|_{L^2(H_\tau)} \bigg) \nonumber\\
&\quad +  k c_{v,\tau} \|v_\tau^-\|_{L^2(H_\tau)} \bigg( \beta \|u^\tau\|_{L^{\infty}(H_\tau)}\|u_\tau^-\|_{L^2(H_\tau)} + \mu_2 \|v^\tau\|_{L^{\infty}(H_\tau)}\|v_\tau^-\|_{L^2(H_\tau)} \bigg) \nonumber\\
&\leq k    \bigg(  \mu_1 c_{u,\tau}\|u^\tau\|_{L^{\infty}(H_\tau)} + \beta c_{u,\tau}\|v^\tau\|_{L^{\infty}(H_\tau)} + \beta c_{v,\tau} \|u^\tau\|_{L^{\infty}(H_\tau)}  \bigg) \|u_\tau^-\|^2_{L^2(H_\tau)} \nonumber\\
&\quad +  k \bigg(  \beta  c_{u,\tau}\|v^\tau\|_{L^{\infty}(H_\tau)} + \beta  c_{v,\tau} \|u^\tau\|_{L^{\infty}(H_\tau)}+ \mu_2 c_{v,\tau}
\|v^\tau\|_{L^{\infty}(H_\tau)}  \bigg) \|v_\tau^-\|^2_{L^2(H_\tau)} \nonumber\\
&\leq  C_0  \left( c_{u,\tau} + c_{v,\tau}  \right) \left( \|u_\tau^-\|^2_{L^2(H_\tau)} +\|v_\tau^-\|^2_{L^2(H_\tau)} \right).
\end{align}
Using  \eqref{E:4.10} and the definition of $H_\tau^w$,  we have $H_\tau^u ,H_\tau^v \subset \mathbb R^2\backslash B_R$, which, combining  \eqref{E:4.9} and the definition of $c_{w, \tau}$, means that
$$c_{u,\tau}+c_{v,\tau} \leq \left( \int_{\mathbb R^2\backslash B_R} (y_1-\tau)^2 u^2(y) \mathrm{d}y \right)^\frac{1}{2} + \left( \int_{\mathbb R^2\backslash B_R} (y_1-\tau)^2 v^2(y) \mathrm{d}y \right)^\frac{1}{2} < \frac{1}{2 C_0}.$$
So $\| u_\tau^- \|^2_{L^2(H_\tau)} +\| v_\tau^- \|^2_{L^2(H_\tau)} =0 $, which implies that $u_\tau^- \equiv 0,\ v_\tau^- \equiv 0 $ in $H_\tau$ for any $\tau\in(t-\varepsilon, t]$. We complete the proof.
\end{proof}

\begin{proof}[\textbf{Proof of Theorem \ref{th:1.1}:}]   From Lemma \ref{lm:4.3}, there exists $T>0$ such that, for any $t>T$,
\begin{align}\label{E:4.13}
u_t(x) \geq 0, \  v_t(x) \geq 0  \ \ \mbox{in} \ H_t.
\end{align}
Starting from such a $t>T$, one can move the plane $x_1=t$ to the left as long as \eqref{E:4.13} holds.
Suppose that there exists a $t_0>0$ such that $u_{t_0}(x), v_{t_0}(x) \geq 0$ in $H_{t_0}$,
but $u_{t_0}\not\equiv0$ or $v_{t_0}\not\equiv0$ in $H_{t_0}$.
By Lemma \ref{lm:4.5},  there exists a $\varepsilon > 0$ such that,  for any $\tau\in(t_{0} -\varepsilon, t_{0}]$,
$$ u_\tau (x) \geq 0,\ v_\tau (x) \geq 0 \ \ \mbox{in} \ H_\tau .$$
Using Lemma \ref{lm:4.4}$(ii)$, we  have $u_\tau \equiv 0 \ \mbox{in} \ H_\tau$ if and only if $v_\tau \equiv 0$ in $H_\tau$. So we obtain that if the process of moving plane stops at $t_1 $,
then  $u_{t_1} \equiv 0, v_{t_1} \equiv 0$ in $H_{t_1}$  and $u_{t}\geq 0, v_{t}\geq 0$ in $H_t$ for any $t\geq t_1$.

By a translation, we may assume that $u(0)=\max\limits_{x\in\mathbb R^2} u(x)$ and $v(0)=\max\limits_{x\in\mathbb R^2} v(x)$.
Therefore,  the process of moving plane from any direction must stop at the origin.
So $u$ and $v$ are  radially symmetric and monotone decreasing.

We complete the proof.
\end{proof}

\section*{Acknowledgments}
We thank to the editor and  the referees for their time and comments.
This work is partially supported by the Natural Science Foundation of China (Grant No. 12061012)
and the special foundation for Guangxi Ba Gui Scholars


\end{document}